\documentclass[12pt]{amsart}
\usepackage[applemac]{inputenc}
\usepackage[english]{babel}

\usepackage{amsmath,amsthm,amssymb,mathrsfs,latexsym,amsfonts}

\usepackage{bbm}

\usepackage{amssymb}
\usepackage{amsmath}

\usepackage{url,graphicx}

\usepackage{graphicx,psfrag,epsfig,multirow}

\usepackage{color}

\def\Z{{\mathbb Z}}\def\T{{\mathbb T}}\def\R{{\mathbb R}}

\def\GG{{\mathcal G}}

\def\cC{{\mathcal C}}

\def\cT{{\mathcal T}}
\def\cD{{\mathcal D}}
\def\cH{{\mathcal H}}

\def\cU{{\mathcal U}}

\def\o{\omega}\def\e{\eta}

\def\be{\beta}
\def\be{\beta}
\def\th{\theta}

\def\i{\infty}

\def\l#1{\langle #1\rangle}

\def\<{\langle}
\def\>{\rangle}

\theoremstyle{plain}
\newtheorem{Thm}{Theorem}
\newtheorem{theo}{Theorem}

\newtheorem{Rem}[Thm]{Remark}
\newtheorem{Lem}[Thm]{Lemma}
\newtheorem{SubLem}[Thm]{Sublemma}

\newtheorem*{Cor*}{Corollary}
\newtheorem*{Prop*}{Proposition}

\newtheorem{Main}{Theorem}

\def\cC{\mathcal C}

\def\e{\varepsilon}
\def\pa{\partial}

\def\th{\theta}

\newcommand{\ba}{\overline{A}}

\newcommand{\cO}{\mathcal{O}}

\def\be{\begin{equation}}
\def\ee{\end{equation}}

\def\ba{{\begin{align}}}
\def\ea{{\end{align}}}

\def\bm{\begin{matrix}}
\def\em{\end{matrix}}

\def\0{{\mathbf 0}}

\newtheorem{thm}{Theorem}

\newtheorem{prop}[thm]{Proposition}

\theoremstyle{remark}

\def\cO{\mathcal{O}}

\def\cC{\mathcal{C}}

\theoremstyle{definition}

\newtheorem{definition}{Definition}

\def\ssm{\smallsetminus}

\renewcommand{\setminus}{\ssm}

\newcommand{\eps}{{\epsilon}}

\newcommand{\om}{{\omega}}

\newcommand{\N}{{\mathbb N}}

\def\B0{{\bold{0}}}

\def\be{\begin{equation}}
\def\ee{\end{equation}}

\catcode`\@=12

\def\Empty{}
\newcommand\oplabel[1]{
  \def\OpArg{#1} \ifx \OpArg\Empty {} \else
  	\label{#1}
  \fi}
		
\newcommand{\comm}[1]{}
\newcommand{\comment}[1]{}

\usepackage{color}

\definecolor{orange}{rgb}{1,0.5,0}

\begin{document}

\title[Isolated elliptic fixed points for smooth Hamiltonians]{Isolated elliptic fixed points for smooth Hamiltonians}
\author{Bassam Fayad and Maria Saprykina}
\address{
B.Fayad: IMJ-PRG CNRS, France; M. Saprykina: Inst. f\"or Matematik,
KTH, 10044 Stockholm, Sweden\\ }
\email{bassam.fayad@imj-prg.fr, masha@kth.se}
\date{\today}
\thanks{The first author is supported by  ANR BEKAM and ANR GeoDyM and  the project BRNUH. This work was accomplished while the first author was affiliated to the Laboratorio Fibonacci of the Scuola Normale  Superiore  di Pisa}


\begin{abstract}
We construct on $\R^{2d}$, for any $d \geq 3$,  smooth Hamiltonians  having an elliptic equilibrium with an arbitrary frequency, that is not accumulated by a positive measure set of invariant tori. For $d\geq 4$, the Hamiltonians we construct have not any invariant torus of dimension $d$. 
Our examples are obtained by a version of the successive conjugation scheme {\it \`a la}  Anosov-Katok.
\end{abstract}

\maketitle

\section*{Introduction} 

KAM theory  (after Kolmogorov Arnol'd and Moser) asserts that generically an elliptic fixed point of a Hamiltonian system is stable in a probabilistic sense, or {\it KAM-stable:  the fixed point is accumulated by a positive measure set of invariant Lagrangian tori}. 
 In classical KAM theory, an elliptic fixed point is shown to be KAM-stable under the hypothesis that the frequency vector at the fixed point is non resonant (or just sufficiently non-resonant) and that the Hamiltonian is sufficiently smooth and satisfies  the Kolmogorov non degeneracy condition that involves its Hessian matrix at the fixed point. Further development of the theory allowed to relax the non degeneracy condition. In \cite{EFKchenciner} KAM-stability was established for non resonant elliptic fixed points under the R\"ussmann non-planarity condition on the Birkhoff normal form of the Hamiltonian.

 The problem is more tricky if no non-degeneracy conditions are imposed on the Hamiltonian. In the analytic setting, no examples are known of an elliptic fixed point with a non-resonant frequency $\omega_0$ that is not KAM-stable or Lyapunov unstable (none of these two properties implies the other).  
 
It was conjectured by M. Herman in his ICM98 lecture \cite{H}
that for analytic Hamiltonians, KAM-stability holds in the neighborhood of a  an elliptic fixed point if its frequency vector is assumed to be Diophantine. The conjecture is known to be true in two degrees of freedom \cite{R}, but remains open in general. Partial results were obtained in \cite{EFKchenciner} and \cite{EFK}.

Below analytic regularity, Herman proved that KAM-stability of a Diophantine equilibrium holds without any twist condition for smooth Hamiltonians in 2 degrees of freedom (see \cite{H2}, \cite{FK} and \cite{EFK}).
In his ICM98 lecture \cite[\S 3.5]{H}, he announced  that  KAM-stability of Diophantine equilibria does not hold for smooth Hamiltonians\footnote{Herman actually raised the problem in the very related context of 
symplectic maps.}  in four or more degrees of freedom, without giving any clew about the counter-examples he had in mind. He also announced that nothing was known about  KAM-stability of Diophantine equilibria  for smooth Hamiltonians in three degrees of freedom. 

In this note, we settle this problem by constructing examples of smooth Hamiltonians  for any $d \geq 3$  having non KAM-stable elliptic equilibria with arbitrary frequency.  We now state our results more precisely.

 Let
$\o_0\in\R^d$  and let
$$
 (*) \quad \left\{\begin{array}{l}
H(x,y)=\l{\o_0,r} +\cO^3(x,y)\\
r=(r_1,\dots,r_d),\quad r_j=\frac12(x_j^2+y_j^2)
\end{array}\right. $$
be a smooth function defined in a neighborhood of $(0,0)$. The Hamiltonian system associated to $H$ is given by the vector field $X_H=(\pa_{y}H, -\pa_{x}H)$, namely
$$  \quad\left\{\begin{array}{l}
\dot x=\pa_{y}H(x,y)\\ 
\dot y=-\pa_{x}H(x,y).\end{array}\right.$$
The flow of $X_H$ denoted by $\Phi_H^t$ has an elliptic fixed point at the origin with frequency vector $\o_0$.

\medskip

In \cite{EFK}, it was shown that for any $\o_0 \in \R^d, d\geq 4$, it is possible to construct  $C^\infty$ (Gevrey) Hamiltonians $H$ with a smooth invariant torus, on which the dynamics is the translation of frequency $\o_0$, that is not  accumulated by a positive measure of invariant  tori. In this note we adapt the latter construction to the context of elliptic equilibria and we extend it to the three degrees of freedom case. 

It is a common knowledge that creating instability in the neighborhood of a fixed point is more delicate than in the context of invariant tori, mainly because the action angle coordinates are singular in the neighborhood of the axes $\{r_i=0\}$.  For instance, when all the coordinates of $\om_0$ are of the same sign, the fixed elliptic point is Lyapunov stable, while it is easy to produce examples of diffusive and isolated invariant tori for any resonant frequency vector, even in the analytic category (see \cite{Sev}).

\begin{definition} 
We say that $\Phi^{t}_{H}$ is {\it diffusive} if given any $A>0$ there exists $p$ and $t_1,t_2 \in \R$ such that $|\Phi^{t_1}_{H}(p)|\leq A^{-1}$ and $|\Phi^{t_2}_{H}(p)|\geq A$. 
\end{definition}

Obviously, if the flow is diffusive, the origin is not Lyapunov stable.

\begin{Main} \label{theorem.main4}  For any $\om_0 \in \R^d, d \geq 4$, there exists $H \in C^{\infty} (\mathbb R^{2d})$ as in $(*)$,
such that $\Phi^t_H$ has no invariant torus of dimension $d$. More precisely, the manifolds $\{r_i=0\}$ for $i\leq d$, are foliated by invariant tori of dimension $\leq d-1$ and all other obits accumulate on these manifolds or at infinity. 

Moreover, if the coordinates of $\om_0$ are not all of the same sign, then $\Phi^t_H$ is diffusive. 
\end{Main}

In the case $d=3$, our examples will have invariant Lagrangian tori of maximal dimension (equal to $3$) that accumulate the origin, but only for $r_3$ in a countable set.  

\begin{Main} \label{theorem.main3}  For any $\om_0 \in \R^3$, there exists $H \in C^{\infty} (\mathbb R^{6})$  as in $(*)$, and a sequence $\{a_n\}_{n \in \Z}$ of real numbers such that $\displaystyle  \lim_{n \to -\infty} a_n =0$
and $\displaystyle \lim_{n \to +\infty} a_n =+\infty$, such that the manifolds
$\{r_3=a_n\}$, as well as
$\{r_i=0\}$ for $i\leq 3$,   are foliated by invariant tori and such that 
all other obits accumulate on these manifolds or at infinity. 

  Moreover, if the coordinates of $\om_0$ are not all of the same sign, then $\Phi^t_H$ is diffusive. 
\end{Main}

\begin{Rem} {\rm
The same construction can be carried out for invariant quasi-periodic tori and gives examples of KAM-unstable tori with arbitrary frequency in 3 degrees of freedom.  } \end{Rem}

\begin{Rem} {\rm  Our examples are obtained by a successive
    conjugation scheme {\it \`a la}  Anosov-Katok \cite{AK},
and the flows  that we obtain are rigid in the sense that their iterates along a subsequence of time converges to identity in the $C^\infty$ topology. } \end{Rem}  

\begin{Rem} {\rm
In  case all   the coordinates of $\om_0$ are of the same sign, there are naturally no diffusive orbits since
the equilibrium is Lyapunov stable. } \end{Rem}

In the case where not all the components of $\om_0$ are of the same sign and $d\geq 3$, Douady gave in \cite{douady} examples of elliptic fixed points with diffusive trajectories. However, his construction, that produces actually examples with an arbitrarily chosen Birkhoff normal form  at the fixed point, does not overrule  KAM-stability.

\section{Notations} 

\noindent --A vector  $\o_0 \in \R^d$ is said to be non-resonant if for any $k
\in \Z^d\setminus \{0\}$ we have that $| \left< k,\o \right> |\neq 0$,
where $\left< \cdot ,\cdot \right>$ denotes the usual scalar product.

\noindent --A vector $\o_0 \in \R^d$ is said to be Diophantine if there exist $N>0$ and $\gamma>0$  such that for any  $k
\in \Z^d\setminus \{0\}$ we have that $|\left< k,\o_0 \right>|\geq \gamma \|k\|^{-N}$.

\noindent --A non-resonant vector $\o_0$ is said to be Liouville if it is not Diophantine.

\noindent --We denote by  $\cO^T_H(p)$ the
orbit of length $T$ of the point $p$ by the Hamiltonian flow of $H$. The full orbit of $p$ is denoted by  $\cO^\infty_H(p)$.

\noindent --The notation of type $\{r_i<A\}$ should be understood as
 $\{(r,\th)\mid r_i<A\}$.

\noindent --We shall say, with a slight abuse of notation, that ${\cO^\infty_H(p)}$ 
accumulates on $\{r_i=\infty\}$ for some $i=1,\dots ,d$ if $\text{proj}_{r_i} {\cO^{t_j}_H(p)} \to
\infty$ over a sequence of times $(t_j)$.

\section{Orbits accumulating the axis and diffusive orbits} 

\subsection{Two degrees of freedom.} As discussed earlier, in $2$ degrees of freedom it follows from the Last Geometric Theorem
of Herman  (see \cite{H2}, \cite{FK} and \cite{EFK}) that if $\o_0 \in \R^2$ is Diophantine then
if $H \in C^{\infty} (\mathbb R^{4})$ is as in $(*)$, then the origin is KAM-stable for $\Phi^t_H$. 

We will be interested in constructing close to integrable
non KAM-stable 
(and diffusive if the frequency vector satisfies $\o_{0,1}\o_{0,2}<0$) examples in 2 degrees of freedom when the frequency is Liouville. Consider 
\begin{equation}\label{def_H02}
H_0^2(r)=\l{\omega_0,r}.
\end{equation}
Let $\cU_2$ be the set of
symplectomorphisms $U$ such that $U(r,\th)=(r,\th) $  near 
the axes $\{r_i=0\}$, $i=1,2$, as well as for $|r|$ sufficiently
large.
Consider the class of conjugates of $H_0^2$ 
by the elements of $\cU_2$, and denote its $C^\infty$-closure by $\bar{ \cH_2}$.

\begin{theo} \label{theo.gdelta2} Let $H_0^2(r)=\l{\omega_0,r}$ for a 
Liouville vector $\o_0 \in \R^2$. 
Let $\cD$ be the set of Hamiltonians $H \in \bar{\cH_2}$ such that 
for each $p \in \R^4$ 
we have that 
$\cO^\infty_H(p)$ 
accumulates on at least one of the following sets: $\{|r|=\infty\}$,
$\{r_{1}=0\}$ or $\{r_{2}=0\}$.
 
 If $\o_{0,1} \o_{0,2}<0$, then we assume moreover that for $H \in
 \cD$, for every $A>0$, there exists 
 $p'\in \R^4$ such that
$\cO^\infty_H(p')$ intersects both  $\{|r|\leq A^{-1} \}$ and $\{|r| \geq A\}$.

Then $\cD$ contains  a dense (in the $C^\infty$ topology) $\GG^\delta$ subset of $\bar{\cH_2}$. 
\end{theo}

\begin{Rem} {\rm  Note that for $H \in \cD$ and for $i=1,2$, the sets
    $\{ r_i=0\}$ are foliated by invariant tori of dimension $1$ on which
    the dynamics is the rotation by angle $\o_{0,j}$ $j\in
    \{1,2\}\setminus \{i\}$. These are the only invariant tori for
    $H$.
} 
\end{Rem}  

\begin{Rem} {\rm  The construction can be extended to any degrees of
freedom $d\geq 2$ and any Liouville vector $\o \in \T^d$.  
} \end{Rem}

\subsection{Four degrees of freedom and higher.}

We consider $d=4$, the case $d\geq 5$ being similar. Fix $\o_0 \in \R^4$.
To prove Theorem \ref{theorem.main4} we will use the same technique of construction that serves in the Liouville $2$ degrees of freedom construction. 

We will first introduce a
completely integrable flow with a fixed point at the origin of frequency $\o_0$ following \cite{EFK}. It will have the form 
\begin{equation} \label{def_H4}
H_0^4(r)=\l{\omega (r_4), r}= \l{\omega_0 +f(r_4), r} ,
\end{equation}
where we use action coordinates  $r_j(x,y)=(x_j^2+y_j^2)/2$ and  where
$$
f(r_4)= (f_1(r_4),f_2(r_4),f_3(r_4),0)
$$ 
with $f$  defined as follows.

We call a sequence of intervals (open or closed or half-open)
$I_n=(a_n,b_n) \subset (0,\i )$ 
an increasing cover of the half line if:
\begin{enumerate} 
 \item  $\displaystyle \lim_{n \to -\infty} a_n =0$
 \item  $\displaystyle \lim_{n \to +\infty} a_n =+\infty$
\item $a_n\leq b_{n-1}<a_{n+1}\leq b_n$.
\end{enumerate}

\begin{prop}\cite{EFK} \label{colimacon} Let
  $(\omega_{0,1},\omega_{0,2},\omega_{0,3}) \in \R^3$ be fixed. 
For every $\eps>0$ and every $s\in \N$, there exist an increasing
cover $(I_n)$ of $(0,\i)$ and functions $f_i \in C^\infty(\R,[0,1])$,
$i=1,2,3$, such that $\|f_i\|_s<\eps$ and $ f_i(0)=0$,  and
\begin{itemize} 

\item For each  $n\in \Z$, the functions $f_1$ and $f_2$ are  constant on $I_{3n}$ :  
$${f_1}_{| I_{3n}}\equiv \bar{f}_{1,n}, \quad {f_2}_{| I_{3n}}\equiv \bar{f}_{2,n}$$ 
\item For each  $n\in \Z$, the functions $f_1$ and $f_3$ are  constant on $I_{3n+1}$ :
$${f_1}_{| I_{3n+1}}\equiv \bar{f}_{1,n}, \quad {f_3}_{| I_{3n+1}}\equiv \bar{f}_{3,n}$$ 
\item For each  $n\in \Z$, the functions $f_2$ and $f_3$ are  constant on $I_{3n-1}$ :
$${f_2}_{| I_{3n-1}}\equiv \bar{f}_{2,n}, \quad {f_3}_{| I_{3n-1}}\equiv \bar{f}_{3,n-1}$$ 
\item The vectors $(\bar{f}_{1,n}+\omega_{0,1},\bar{f}_{2,n}+\omega_{0,2})$, $(\bar{f}_{1,n}+\omega_{0,1},\bar{f}_{3,n}+\omega_{0,3})$ and  $(\bar{f}_{2,n}+\omega_{0,2},\bar{f}_{3,n}+\omega_{0,3})$ are Liouville.
\end{itemize}

If $\o_{0,1} \o_{0,2}<0$, we ask that $\eps>0$ be sufficiently small so that $\o_1 (r_4)\o_2(r_4)<0$ for every $r_4$. 
\end{prop}

\begin{Rem} {\rm 
It follows that  $f_1,f_2,f_3$ are $\cC^{\i}$-flat at zero.
} \end{Rem}

Notice that, as a consequence of Proposition \ref{colimacon}, for
$r_4\in I_n$ two of the coordinates of
$(f_1(r_1)+\o_1,f_2(r_1)+\o_2,f_3(r_1)+\o_3)$ are constant and form a Liouville
vector. This is why we will be able to use a similar construction as in the two dimensional Liouville case.

Let $\cU_4$ be the set of exact symplectic diffeomorphisms of $\R^8$
with the following properties: $U(r,\th)=(r,\th)$ in the neighborhood of the axes $\{r_i=0\}$,
$i=1,\dots ,4$, as well as for $|r|$  sufficiently large,  and 
$U(r,\th)=(R,\Theta)$ satisfies $R_4=r_4$. 
Let $\cH_4$ be the set of Hamiltonians of the form $H_0^4\circ U$, 
$U\in \cU_4$. Finally we denote $\bar{\cH_4}$ the closure in the
$C^\infty$ topology of $\cH_4$.

We denote 
$$
\widetilde{I_n}=\R^3\times I_n \times  \T^4.
$$
For $H \in \cH_4$ the flow $\Phi_H^t$ leaves $r_4$ invariant.
In particular,  for $U\in \cU_4$ we have
$U(\widetilde{I_n})=\widetilde{I_n}$ 
for any $n \in \Z$.   We
shall show how to make arbitrarily small perturbations
of $H_0^4$   
inside $\cH_4$
that create oscillations of the
corresponding flow in two of the three directions $r_1,r_2,r_3$.
These perturbations will actually be compositions inside $\cH_4$ by
exact symplectic maps obtained from  suitably chosen generating
functions.

Iterating the argument gives a construction by successive conjugations
scheme similar to \cite{AK}. The difference here is that the
conjugations will be applied in a "diagonal" procedure to include more
and more intervals $I_n$ into the scheme. Rather than following this
diagonal scheme which would allow to define the conjugations
explicitly at each step, we will actually adopt a $\GG^\delta$-type
construction (see \cite{FH}) that makes the proof much shorter and
gives slightly more general results.

\begin{theo} \label{theo.gdelta} Let $H_0^4(r)=\l{\omega (r_4), r}$ be
as in (\ref{def_H4}). Let $\cD$ be the set of Hamiltonians $H \in \bar{\cH_4}$ such that 
for each $p \in \R^8$ such that $r_i(p)\neq 0$ for every
$i=1,\ldots,4$,  there exists $i \in \{1,2,3\}$ 
such that ${\cO^\infty_H(p)}$ 
accumulates on at least one of the sets $\{r_{i}=\infty\}$ and 
$\{r_{i}=0\}$.

If $\o_{0,1} \o_{0,2}<0$, we assume moreover that for $H \in \cD$, there exists for every $A>0$, $p'\in \R^8$ such that
$\cO^\infty_H(p')$ intersects both  $\{|r|\leq A^{-1} \}$ and 
$\{|r|\geq A\}$. 

Then $\cD$ contains  a dense (in the $C^\infty$ topology) $\GG^\delta$ subset of $\bar{\cH_4}$. 
\end{theo}

\begin{proof}[Proof that Theorem \ref{theo.gdelta} implies Theorem
  \ref{theorem.main4}]  Note that for $H \in \cD$ and $i=1,\dots ,4$, 
the set $\{r_i=0\}$ is foliated by invariant tori of dimension $3$ on which the dynamics is that of the integrable Hamiltonian $H_0^4$. We want to show that these are the only invariant tori of $H$.  Indeed, let $p \in \R^8$ such that $r_i(p)\neq 0$ for every
$i=1,\ldots,4$. Since the orbit of $p$ accumulates on the axis or at infinity, it cannot lie on an invariant compact set.

Note now that if  $\o_0$ does not have all its coordinates of the same sign, we
can assume that $\o_{0,1}\o_{0,2}<0$ by possibly renaming the variables. 
Hence the second part of Theorem \ref{theorem.main4} follows form the second part of Theorem  \ref{theo.gdelta}.
\end{proof}

\subsection{Three degrees of freedom}

The construction of Theorem \ref{theorem.main3} for $d=3$ will be
similar to the case $d=4$ but with this difference that we cannot
count anymore on an invariant action variable $r_4$ that played the
role of a parameter. Instead, one of the action coordinates will both
involved in the diffusion, and play the role of the parameter. 
We choose this variable to be $r_3$ and assume without loss of
generality that if the coordinates of $\o_0$ are not all of the same
sign then $\o_{0,1}\o_{0,2}<0$.

We   fix a sequence of intervals $I_n=[a_{n-1},a_{n}] \subset
(0,\infty)$, $n\in \Z$, such that  
$\displaystyle \lim_{n \to -\infty} a_n =0$,
and  $\displaystyle \lim_{n \to +\infty} a_n =+\infty$,
and let $\widetilde I_n= \R_+^2\times I_n \times \T^3$.
We introduce a
completely integrable flow with a fixed point at the origin by
\begin{equation} \label{H03} 
H_0^3(r)= \l{\omega(r_3) , r} =   \l{\omega_0+f(r_3),r}, 
\end{equation}
where 
$f=(f_1,f_2,f_3)$ is as in Proposition \ref{colimacon} with $r_4$
replaced by $r_3$, and  we use action-angle coordinates as above.

Let $\cU_3$ be the set of
symplectomorphisms $U$ such that $U(r,\theta)=(r,\theta)$  near 
the axes $\{r_i=0\}$, $i=1,2,3$,  as well as near the sets
$\{r_3 = a_n\}$, $n\in\Z$, and for $|r|$ sufficiently large.
Consider the class of conjugates of $H_0^3$ 
by the elements of $\cU_3$, and denote its $C^\infty$-closure by $\bar{ \cH_3}$.

\begin{theo} \label{theo.gdelta_d3} Let $H_0^3(r)=\l{\omega (r_3), r}$ be
  as in \eqref{H03}. Let 
$\cD$ be the set of Hamiltonians $H \in \bar{\cH_3}$ such that 
for each $p \in \R^6$ satisfying $r_i(p)\neq 0$ for every
$i=1,2,3$ and  $r_3(p) \notin \{a_n\}_{n\in \Z}$ 
we have  that 
${\cO^\infty_H(p)}$ accumulates 
on at least one of the following
sets: 
$\{r_{1}=\infty\}$, $\{r_{2}=\infty\}$,
$\{r_{1}=0\}$, $\{r_{2}=0\}$, $\cup_{n \in \Z} \{r_3=a_n \}$.
 
 If $\o_{0,1} \o_{0,2}<0$, then we assume moreover that for $H \in \cD$, there exists for every $A>0$, $p'\in \R^6$ such that
$\cO^\infty_H(p')$ intersects both  $\{|r|\leq A^{-1} \}$ and
$\{r_{1}\geq A \} \cap \{r_{2}\geq A \}$.

Then $\cD$ contains  a dense (in the $C^\infty$ topology) $\GG^\delta$ subset of $\bar{\cH_3}$ 
\end{theo}

\begin{proof}[Proof that Theorem \ref{theo.gdelta_d3} implies Theorem
  \ref{theorem.main3}]  Note that the  axes and the sets
  $\{r_3=a_n\}$ 
are foliated by invariant tori.  The orbit of $p  \in \R^6$ satisfying $r_i(p)\neq 0$ for every
$i=1,2,3$ and  $r_3(p) \notin \{a_n\}_{n\in \Z}$  cannot accumulate on
any of these sets if it lies on an invariant torus. Hence the only invariant tori for $\Phi^t_H$ are those foliating the axis and $\cup_{n \in \Z} \{r_3=a_n \}$.
The second part of Theorem \ref{theorem.main3} follows clearly from the second part of  Theorem \ref{theo.gdelta_d3}. 
\end{proof}

\section{Proof for the case $d=2$}

All our constructions will be derived from the main building block  with 
two dimensional Liouville frequencies.   The construction is
summarised in the following Proposition \ref{lemma.conjugacy} from
which Theorem \ref{theo.gdelta2} will easily follow.

For $A>0$, denote
$$
R(A):=[A^{-1}, A] \times [A^{-1}, A], \quad \widetilde R(A)
=R(A)\times \T^2.
$$
We define the "margins" by: 
$$
M(A)=\{r_1>A\} \cup \{r_2>A\} \cup \{r_1<A^{-1}\} \cup \{r_2<A^{-1}\}.
$$
We shall refer to the individual sets of the above union as {\it margin sets}.
\begin{prop}\label{lemma.conjugacy}
For any  Liouville vector $  \om=(\om_1,\om_2)$, 
  any $\eps>0, s \in \N $, $A_0>0$, and any symplectic map $V$ that is identity 
outside $\widetilde R(A_0)$ we have that for any $A>A_0$ there exist
  $U\in \cU$ and $T>0$ with the following properties for $H=H_0^2\circ U^{-1} \circ V^{-1}$:
\begin{enumerate}
\item $U=$Id in the complement of $\widetilde R(2A)$,

\item ${\|H-H_0^2\circ V^{-1} \|}_s<\eps$,

\item  For any $P \in \widetilde R(A)$ we have:  $\cO^T_H(P)$ intersects $M(A)$.

\item Moreover, if  $\o_1 \o_2 <0$, 
then there exists  $p' \in \R^4$ 
such that $\cO^T_H(p')$ intersects both $\cap_{i=1}^2
\{r_{i}<2A^{-1}\}$ and $\cup_{i=1}^2 \{r_{i}>A\}$.
\end{enumerate} 
\end{prop}

\begin{proof}[Proof of Theorem \ref{theo.gdelta2}]  For $n,A,T \in \N^*$, let 
$$
\cD(A,T) := \left\{  H \in  \bar{\cH_2}  \mid 
\forall P\in  \widetilde R(A),  
  \cO^T_H(P)   \text{ intersects }  M(A)
\right\}  .
$$
It is clear that $\cD(A,T)$ are open subsets of  
$\bar{\cH_2}  $ in any $C^s$ topology. Proposition \ref{lemma.conjugacy} (1)--(3) implies that $ \cup_{T \in \N^*}   \cD(A,T)$
is dense in  $ \bar{\cH_2}  $ in any $C^s$ topology.
Hence the following set $\bar{\cD} \subset \cD$ is a dense $G^\delta$ set (in any $C^s$ topology) in $\bar{\cH}_2$ 
$$
\bar{\cD} = \bigcap_{A \in \N^*} \bigcup_{T \in \N^*}  \cD(A,T).
$$

In case $\o_{1}\o_{2}<0$ we just have to add to the definition of $\cD(A,T)$  the existence of a point 
 $p' \in \R^4$ 
such that $\cO^T_H(p')$ intersects both $\cap_{i=1}^2
\{r_{i}<2A^{-1}\}$ and $\cup_{i=1}^2 \{r_{i}>A\}$. The density of  $
\cup_{T \in \N^*}   \cD(A,T)$ then follows from (1)--(4) of
Proposition \ref{lemma.conjugacy}. \end{proof}

The rest of this section is devoted to the proof of Proposition
\ref{lemma.conjugacy}. The idea is to construct a conjugacy $U$ that
"wiggles" the invariant tori of $H_0^2$ and makes  them accumulate on
the margin sets. 
Since $V$ in the statement of the proposition is assumed to be
identity outside $\widetilde{R}(A)$ it will be possible to conclude
from there that the Hamiltonian $H=H_0^2 \circ U^{-1} \circ V^{-1}$
satisfies the requirements of the proposition.

One has to observe however that since we want $H_0^2 \circ U^{-1}$ to be very close to $H_0^2$, the "wiggling" will take place almost inside the energy levels of $H_0^2$. In particular, one does not get diffusion in the case $\omega_1 \omega_2>0$ because in that case the energy lines are compact segments (see Figure 1). To be more precise,  fix $\om=(\om_1,\om_2)$ and define the energy line
$$
E_{p}:= \{(r_1,r_2) \in \R^2_+:  
\omega_1 r_1+\omega_2 r_2=\omega_1 r_1(p)+\omega_2 r_2(p)\}, \quad
\widetilde E_p=E_p\times \T^2.
$$
Clearly, $\widetilde E_p$ has the form $\{ H_0^2=const.\}$, and 
is invariant under the flow of $H_0^2$ (with the same
fixed $\om$). For $p=(r,\th)$, let 
$$
\cT(p)=\{r\}\times \T^2
$$ 
denote the flat torus passing through $p$. This is the invariant torus
of $H_0^2$ passing through $p$. Let $H=H_0^2\circ U^{-1}$ for a
symplectic transformation $U$.  
Then $ U (\cT (p))$ is the invariant torus of $H$ passing through the
point $U(p)$. The main ingredient in the proof of Proposition \ref{lemma.conjugacy} is the following lemma in which we construct a symplectic map $U$
such that $ U (\cT (p))$ will "wiggle" inside $\widetilde E_p$ essentially for a large set of
starting points $p$.

\begin{figure}[htb] 
\centering
\resizebox{!}{8cm}{ 
\includegraphics[width=0.5\linewidth]{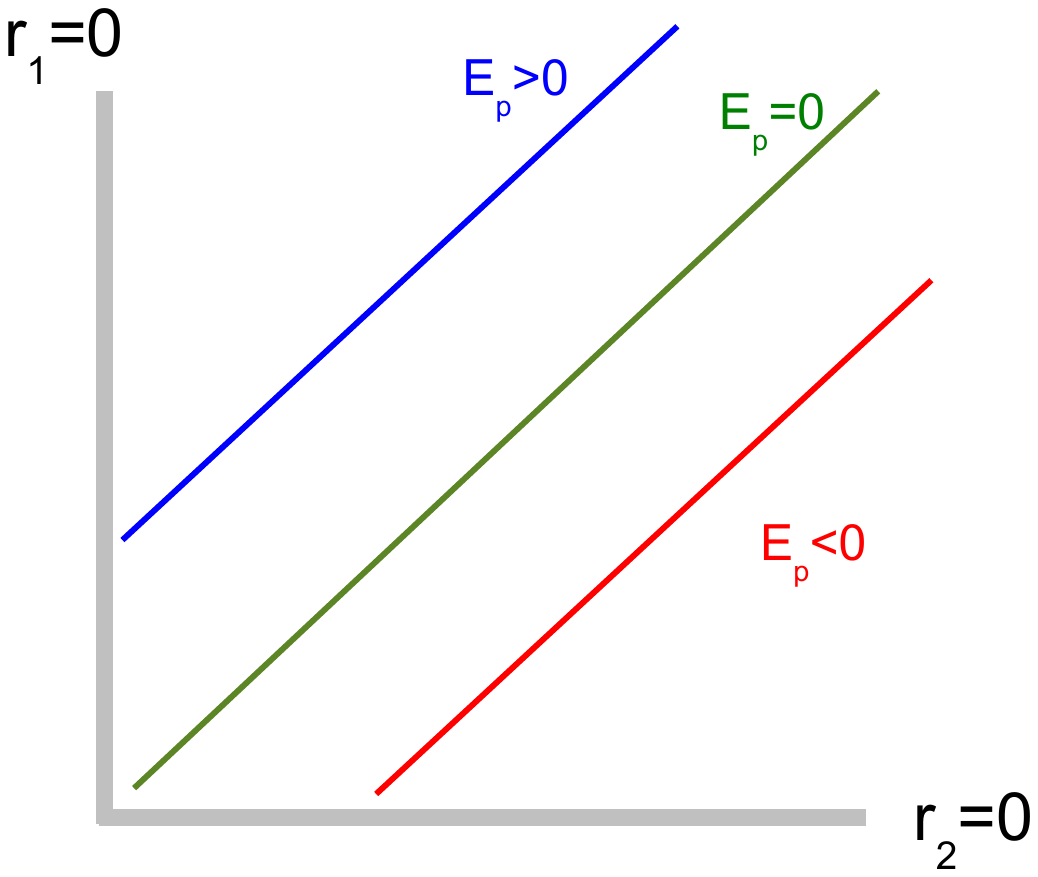} 
\includegraphics [width=0.5\linewidth]{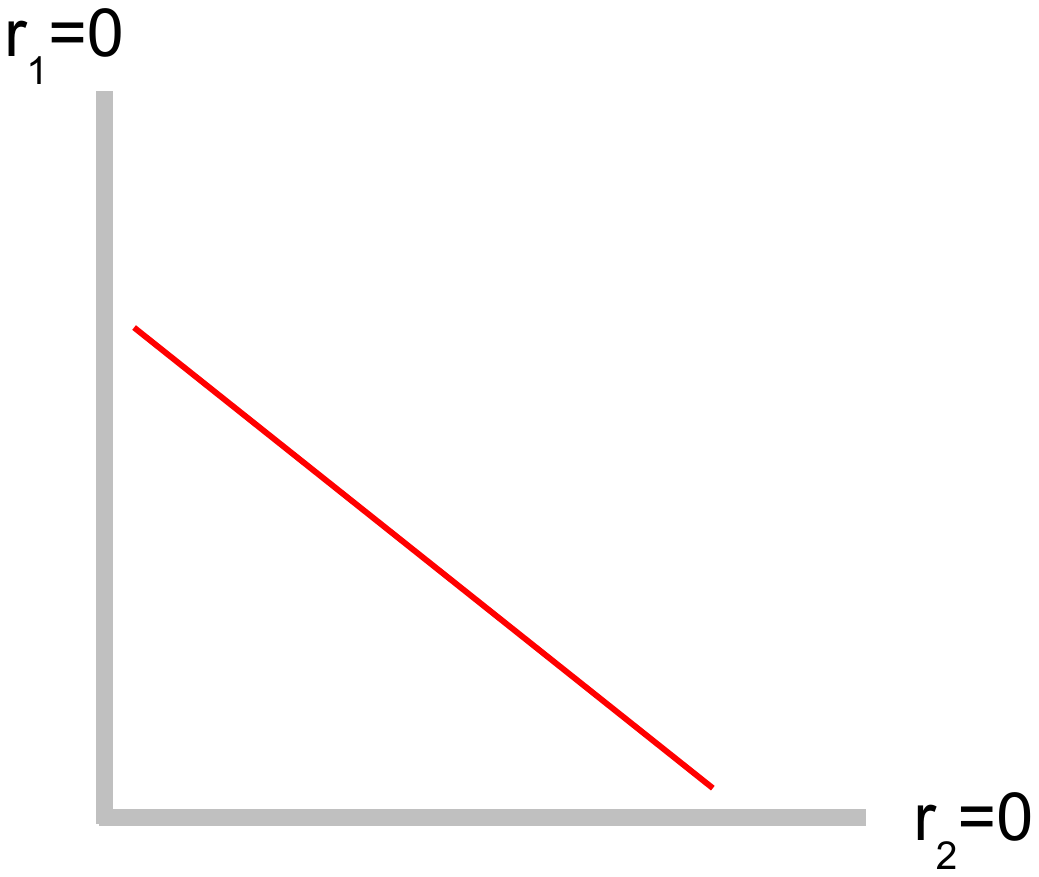} 
}  
\caption{\small The diffusion lines are unbounded in the case
$\omega_1 \omega_2<0$(the figure to the left) and bounded in the
case  $\omega_1 \omega_2>0$.}
\label{orbits}
\end{figure}

\begin{Lem}\label{lemma_U}
For any  Liouville vector $  \om=(\om_1,\om_2)$, 
  any $\eps>0, s \in \N $, $A_0>0$, and any symplectic map $V$ that is identity 
outside $\widetilde R(A_0)$ we have that for any $A>A_0$ there exist
  $U \in \cU$ and $T>0$ with the following properties: 
\begin{enumerate}
\item $U=$Id in the complement of $\widetilde R(2A)$,

\item ${\|H_0^2\circ U^{-1} \circ V^{-1} -H_0^2\circ V^{-1} \|}_s<\eps$

\item  
For any $p\in \R^2_+\times \T^2 $, the torus $V \circ U (\cT (p))$
intersects at least one of the margin sets, and this
intersection is  $\eps$--close to $\widetilde E_p\cap M(A)$.  

\item Moreover, for any $p\in \widetilde R(A)$, the torus
$V \circ U (\cT (p))$ intersects two of the margin sets, and this
intersection is  $\eps$--close to $\widetilde E_p\cap M(A)$.

\end{enumerate} 
\end{Lem}

\begin{proof}[Proof of Proposition  \ref{lemma.conjugacy}. ]
{ Let $U$ be as in Lemma \ref{lemma_U}. The conditions (1) and (2) of
  the proposition follow directly from (1) and (2) of Lemma
  \ref{lemma_U}. The invariant torus of $H= H_0^2\circ U^{-1}\circ
  V^{-1}$ passing
through a point $P=V\circ U(p)$ has the form $V\circ U(\cT(p))$.
It follows from (3) of Lemma \ref{lemma_U} that the latter torus 
intersects $M(A)$. 
Since $\o$ is irrational, the orbit of $p$ is dense on $\cT(p)$. As a
consequence, the orbit of $P$ under the flow of $H$ 
is dense on $V\circ U(\cT(p))$, and conclusion (3) of the proposition
follows.

Let $\o_1\cdot \o_2<0$. In this case, the line $E_{p}$ has a 
positive slope of $\Omega=-\o_1/\o_2$. Take a point $p$ such that
$\max_{i=1,2}\{r_i(p)\} = 2A^{-1}$, and $E_{p}$ passes through
$r=(A^{-1},A^{-1})$. Since $p\in\widetilde R(A)$, the torus
$U(\cT(p))$ intersects two margin sets close to the two components of 
$E_{p}\cap M(A)$. In particular, it contains points in $\{r_1 <
2A^{-1}\} \cap \{r_2 < 2A^{-1}\}$, as well as in 
$\{r_1> A \} \cup \{r_2 > A \}$ (this corresponds to $\widetilde{E}_p$ close to $0$ in Figure 1). The same is true for the torus $V\circ
U(\cT(p))$
since $V$ preserves the margin sets. Of course, this condition holds
for an open set of starting points $p$.  Now, conclusion (4) of the proposition holds for any $p'\in V\circ
U(\cT(p))$ with $p$ as above. }
\end{proof}

We now turn to the proof of Lemma \ref{lemma_U}. We will build the conjugacy $U$ by superposing a large number of very slightly wiggling symplectomorphisms which we now present.

\begin{SubLem}\label{lem_u}
For any {\it Liouville} vector $ \omega$, for any $A>1$, $s>0$, $Q>0$
and $\e > 0$, 
there exists an integer vector $  q=(q_1,q_2)$, $ |q| > Q $, 
and a symplectic map 
$u \in C^{\infty}$ with the following properties:
\begin{itemize} 

\item $u=$Id \ in the complement of $\widetilde R(2A)$.

\item $\| H_0^2 \circ u^{-1} - H_0^2\|_{s}<\e$, 

\item For any $p=(r, \theta) \in \widetilde R(A)$,  
the image $(R,\Theta)= u(r, \theta)$  
satisfies
\begin{equation}\label{lem_u_r}
\begin{aligned}
R_1=r_1 (1 + b r_2 \cos (2\pi \left<   q,  
  \Theta \right>)), \\  
R_2 = r_2 (1 + b (q_2/q_1 )r_1 \cos (2\pi \left<   q,  
  \Theta \right>))
\end{aligned} 
\end{equation}
where $b=(10 A)^{-3} $, and
\begin{equation}\label{lem_u_th}
| \Theta - \theta|_0 < 10^{-3} | q|^{-1}.
\end{equation}
\end{itemize}
\end{SubLem}
These expressions mean that any small ball  
gets stretched by $u$ both in $r_1$- and $r_2$-direction with the
amplitude of order $br_1r_2\geq bA^{-2}$, and large frequency $|q|$
(notice that the amplitude is small for small $r_i$, but bounded from below).  
$U$ will be a composition of 
a large number of such functions $u_j$, $j=1,\dots, N$,
constructed with the same $A$ and $b$,
but decaying $\e_j$ and growing $|  q_j|$.

Here is a heuristic idea. 
Vector $q$ plays two
important roles. 
On the one hand, having large $|q|$, we get high
frequency of oscillation for $R_1$ and good closeness between 
$\theta$ and $\Theta$ (see the formulas above). 

On the other hand, the key estimate  
$\| H_0^2 \circ u^{-1} - H_0^2\|_{s}<\e$ needs $\left<  q, \omega
\right>< \e/pol(q)$, where $pol(q)$ is a polynomial of $q$. It is here
that we use the assumption of $\omega$ being Liouville. It guarantees
that there exists a $q$ with sufficiently large components
providing the desired smallness of $\left<q, \omega \right>$.

\medskip

\begin{proof}[Proof of Sublemma \ref{lem_u}.]  
Let $G(t)\in C^\infty$ be a monotone cut-off function such that
$|G'(t)|\leq 2$ for all $t$ and 
\begin{equation}\label{cutoff}
G(t)=
\begin{cases}
0, \quad \text{for } t\leq 0, \\
 1, \quad \text{for } t > 1. 
\end{cases}
\end{equation}
Denote by $1/A$ the "margin size", and by $b$ the scaling constant:
$b= (10 A)^{-3} $.
Assume without loss of generality that $\Omega=-\om_1 /\om_2 \in (0,1]$.
Define $g(t)= G( 2A t -1)-G(t-A)$.  In this case
$|g'(t)|\leq 4A$ for all $t$, and
$$
g(t)=
\begin{cases}
0, \quad \text{for } t \leq (2A)^{-1}  \text{ or } \ t  \geq A+1, \\
1, \quad \text{for } t \in [A^{-1}, A] . 
\end{cases}
$$ 
Given two integers $q_1$ and $q_2$, whose choice will be specified
later,
we define the symplectic map $u:(  r,  \theta)\mapsto ( 
R,  \Theta) $
by a generating function
$$
S(  r,  \Theta) = \left<   r,   \Theta\right> + 
\frac{b}{2\pi q_1} r_1 r_2 g(r_1) g(r_2) \sin (2\pi \left<   q,  
  \Theta \right>).
$$
It satisfies
$$
\begin{aligned}
R_1= \frac{\partial S (r, \Theta)}{\partial \Theta_1} & 
=r_1(1+ b r_2 g(r_1) g(r_2) \cos (2\pi \left<   q,   \Theta \right>)), \\
R_2 =\frac{\partial S (r, \Theta)}{\partial \Theta_2 } & 
=r_2 (1 + b\frac{q_2}{q_1 } r_1 g(r_1) g(r_2) \cos (2\pi \left<   q,   \Theta \right>)), \\
\th_1 =\frac{\partial S (r, \Theta)}{\partial r_1}& 
= \Theta_1 + \frac{b}{2\pi q_1} \left( r_1 g(r_1) \right)'_{r_1}  
r_2 g(r_2)  \sin (2\pi \left<  q,   \Theta\right>)\\
\th_2 =\frac{\partial S (r, \Theta)}{\partial r_2}& =
\Theta_2 + \frac{b}{2\pi q_1} r_1 g(r_1) \left(r_2 g(r_2) \right)'_{r_2} \sin (2\pi \left<  q,   \Theta\right>).
\end{aligned}
$$
To see that $S$ defines a diffeomorphism it
is enough to verify that the following determinant does not vanish:
$$
\begin{aligned}
&   \text{det} \left( \frac{\partial   R}{\partial   r } \right)=
  \text{det} \left( \frac{\partial   \th}{\partial   \Theta } \right)
 = \text{det} \left( \frac{\partial^2 S (r, \Theta)}{\partial
    r \partial \Theta} \right) = \\
&1 + b\left(   \left( r_1 g(r_1) \right)'_{r_1} r_2 g(r_2)  + 
     \frac{q_2}{q_1} r_1 g(r_1) \left(r_2 g(r_2) \right)'_{r_2}\right)
 \cos(2\pi \left<   q,   \Theta \right>) \\  
&\geq 1- 40 b A^3>0
\end{aligned}
$$ 
by the choice of $b=(10 A)^{-3} $. 

By a local inverse function theorem, $\Theta$ can be expressed as a function
of $(r, \theta)$. We get $\Theta_i=\theta_i+O(|  q|^{-1})$.  
Plugging in this expression into the first two lines, we get a formula
for $u$ in terms of $(  r,   \theta)$. The inverse $u^{-1}$
exists by the same argument.

\medskip

Since $u$ is a diffeomorphism and equals identity in a 
neighborhood of 
$\{r_1=0 \}$ and $\{r_2=0 \}$, we get that the image of $u$ satisfies $R_1>0$ and 
$R_2>0$.
\medskip 

Here we estimate $\| H_0^2\circ u^{-1} (r,\theta)-H_0^2 \|_s$. 
For the above coordinate change we have 
$$
\begin{aligned}
H_0^2\circ u (r,\theta)-H_0^2(r,\theta) = \left< R,  \omega\right> -\left< r,\omega\right>= \\
b \frac{1}{q_1} r_1g(r_1) r_2g(r_2)
\left< \omega , q\right>  \cos (2\pi \left< q, \Theta \right>),
\end{aligned}
$$ 
where $ \Theta = \Theta (r,\theta)$.
We can estimate 
$$
\| H_0^2\circ u^{-1} -H_0^2 \|_{s} =\| (H_0^2 -H_0^2\circ u) \circ u^{-1} \|_{s}\leq F(q, A, s)
\cdot\left<   \omega ,  q\right>, 
$$ 
where $F(q, A,s)$ is a polynomial of $q_1, q_2$ whose degree
depends only on $s$, and the coefficients are bounded by
functions of $A$ and $s$. 
Since vector $\omega$ is Liouvillean,
there exists (an infinite number of) $q$ such that 
$\left< \omega, q \right> < \e/ F(q, A,s) $, 
and the desired estimate follows.
\end{proof}

\begin{proof}[Proof of Lemma \ref{lemma_U}. ]
Fix $V$, $s$, $A_0$ and $\e>0$. Since ($V-$Id) is
compactly supported inside $\widetilde R(A_0)$, the same holds for any $A>A_0$. 
Let  $b $ be as in Sublemma \ref{lem_u}, $\Omega=|\om_1 /\om_2|$ assumed WLOG
to lie in $ (0,1]$, and let $N$ be such that 
\begin{equation}\label{choice_N}
(1 + b /(4A))^N > A^{2}, \quad
(1 - b /(4A))^N < 1/A^2.
\end{equation}
We shall define $u_j$, $j=1,\dots ,N$, by Sublemma \ref{lem_u}
inductively in $j$. 
We choose $q_j$ in the construction of $u_1$ so
that 
$$
|q_j|>|q_{j-1}|^3.
$$
Moreover, for  $j=1,\dots ,N$, $q_j$ are such that $u_j$ satisfies a
(much stronger) condition
\begin{equation}\label{est_uj}
\|\left( H_0^2 \circ u_j^{-1}- H_0^2 \right) 
\circ u_{j-1}^{-1} \circ \dots \circ u_1^{-1} \circ V^{-1}\|_{s}<2^{-j} \e.
\end{equation} 
Recall from the proof of Sublemma \ref{lem_u} that this is done 
by choosing vector $q_j$ at each step so that 
$\left< \omega , q_j\right>$ is sufficiently small depending
on $q_1,\dots,q_{j}, s, A $ and $\eps$. 
Define 
$$
U=u_1\circ \dots \circ u_N .
$$ 
Then  the first statement of the lemma holds since each $u_j$ 
is identity outside $\widetilde R(2A)$ by construction.
The second one follows from (\ref{est_uj}) and the triangle inequality:
$$
\begin{aligned}
&\| \left( H_0^2\circ U^{-1} - H_0^2 \right)\circ V^{-1}\|_s \leq \\
&\sum_{j=1}^N  
\| \left( H_0^2\circ u_j^{-1} - H_0^2 \right)     \circ u_{j-1}^{-1} \circ
\dots \circ u_0^{-1}\circ V^{-1} \|_{s} 
<  \e \sum_{j=1}^N 2^{-j} < \e,
\end{aligned}
$$
where $u_0 =id$ for the uniformity if notations.

To prove (3), fix $p^0=(r^0,\th^0)\in \widetilde R(A)$, and consider 
the flat torus $\cT_{p^0}$ passing 
through $p^0$. Then the invariant torus
of $H_0\circ U^{-1}$ passing through $P^0=U(p^0)$ has the form 
$$
U (\cT_{p^0}) = u_1\circ \dots \circ u_N(\cT_{p^0}).
$$
It lies on the invariant surface $\{ H_0^2\circ U^{-1}(p) = const \}$, which is
$\e$-close to $\widetilde E_{p^0}$, due to (2). 

Since $V$ is identity on the margin sets, we just have to show that  $U (\cT_{p^0})$ intersects the two margin sets.  We
assume without loss of generality that
$\theta_1(p^0)=0$. Given $\th \in \{0\} \times \T$ we define 
$(r^{1},\th^{1})=u_N(r^{0}, \th^{0}) $, $(r^{2},\th^{2})=u_{N-1}(r^{1},\th^{1}) $, 
and in general
\begin{equation}\label{notations_l8}
(r^{j+1},\th^{j+1})= u_{N-j} (r^{j},\th^{j})= u_{N-j}\circ \dots
u_N(r^{0},\th^{0}), 
\quad j=0,\dots N-1.
\end{equation}
The lower index indicates the component:
$(r^{j},\th^{j})=(r^{j}_1,r^{j}_2, \th^{j}_1,\th^{j}_2)$.

So, (3) follows if we prove the following two claims.

{\bf Claim 1.} There exists $\hat{\theta} \in \{0\} \times \T$ and $\check \theta\in \{0\} \times \T$ such that
\begin{align} \label{cosplus} 
\cos(2\pi
\left<\hat \th^{j+1}, q_{N-j} \right>)&>1/2 \ \text{ for all } \
j=0,\dots N-1 \\
\label{cosmoins} 
\cos(2\pi
\left<\hat \th^{j+1}, q_{N-j} \right>)&<-1/2 \ \text{ for all } \
j=0,\dots N-1 
\end{align}

\bigskip 

{\bf Claim 2.}   $U(r,\hat \th)$ and $U(r,\check \th)$ lie close to the two different ends of $\widetilde E_{p^0} \cap R(A)$

\bigskip

\noindent {\it Proof of Claim 1.}   
We  
will actually  restrict ourselves to an interval $K(p^0)$ 
in $\cT_{p^0}$ and show that even 
$U(K(p^0))$ intersects two margin sets. We do so in order to use
one dimensional calculus.

So,
assume without loss of generality that
$\theta_1(p^0)=0$ and let
$$
K(p^0):= \{r^0\}\times \{0\}\times \T
$$ 
be an {\it interval in $\th_2$-direction} passing through $p^0$.

We shall present two subsets of $K(p^0)$, that we denote by
$p^0\times\hat L_0$ and 
$p^0\times\check L_0 $, 
such that $U(p^0\times\hat L_0) $ intersects 
$\{r_1\geq A\} \cup \{r_2\geq A\}$, and  
$U(p^0\times\check L_0) $ intersects 
$\{r_1\leq A^{-1}\} \cup \{r_2\leq A^{-1}\}$.

Recall the notations from (\ref{notations_l8})
By Sublemma \ref{lem_u}, if $r^{j}\in R(A)$, we have: 
$r^{j+1}_1=r^{j}_1( 1 + br^{j}_2 \cos(2\pi
\left<\th^{j+1}, q_{N-j} \right>))$ for all $j=0,\dots N-1$. 

The simplified idea is the following. Imagine that $\th^{j+1}=\th^0$
for all $j$. Since; by construction, 
the frequencies $|q_{j}|$ grow very fast with $j$, there are many
points where  $\cos(2\pi \left<\th^{0}, q_{j} \right>))>1/2$ for all
$j\leq N$. This implies that $r^{j+1}_1 > r^{j}_1 + br^{j}_1r^{j}_2/2 $ 
for all $j=0,\dots N-1$, and we can prove that $r^{N}_1>A$.

In reality though, $\th^{j+1}$ can be different from $\th^0$ (but
close), and one should be more careful. We want
to describe the set of points $\hat \th^0 \in \{0\}\times \T$  such that for the images of
$(r^0,\hat \theta^0)$ we have:  
$$ 
\cos(2\pi
\left<\hat \th^{j+1}, q_{N-j} \right>)>1/2 \ \text{ for all } \
j=0,\dots N-1.
$$ 
We claim that the set
$$ 
\hat J_{j+1}=\{\th^0\in \{0\}\times \T \mid  \cos(2\pi \left<\hat \th^{j+1}, q_{N-j}
\right>)>1/2 \}
$$
consists of $|q_{N-j}|$ disjoint intervals of size at least
$C/(10|q_{N-j}|)$ 
whose
midpoints are at most $C/|q_{N-j}|$ distant from the nearest
neighbour,
where $C$ is a constant only
depending on $\Omega=-\o_1/\o_2$. 
Since $|q_{j}|$ are assumed to grow very fast with $j$, the above
estimates imply that 
$\hat L_0:=\cap_{j=1}^{N} \hat J_{j}$ is nonempty, which proves Claim 1.

To prove the above estimate, notice that it becomes evident if we replace
$\th^{j+1}$
in the definition of $\hat J_{j+1}$ by $\th^0$. In this simplified case, 
the desired intervals are of size $C/(3|q_{N-j}|)$, and the 
midpoints are $C/|q_{N-j}|$ distant.
By Sublemma \ref{lem_u}, $| \theta^1 - \theta^{0}|_0< 10^{-3}
|q_N|^{-1}$,  $| \theta^2 - \theta^{1}|_0 < 10^{-3} |q_{N-1}|^{-1}$,
and in general
$| \theta^{l+1} - \theta^{l}|_0 < 10^{-3} | q_{N-l}|^{-1}$.
Then
\begin{equation}\label{est(th0->thj)}
| \th^{j+1}- \th^0 |_0 \leq 10^{-3} \left(\frac{1}{|q_{N}|}+\dots + \frac{1}{|q_{N-j}|}\right),
\end{equation}
and 
$|\left<\th^{j+1}, q_{N-j} \right> - \left<\th^{0}, q_{N-j}
\right>| \leq 10^{-3}  (\frac{1}{|q_{N}|}+\dots
+\frac{1}{|q_{N-j}|}) |q_{N-j}| \leq 10^{-2}
$. 
Hence, $\hat L_0:=\cap_{j=1}^{N} \hat J_{j}$ is nonempty.

By the same argument, there exists a nonempty subset 
$\check L_0$ of $\{0\}\times \T^1$ 
 such that  for 
$(r^0,\check \theta^0)$ with $\check \theta^0 \in \check L_0$ we have:  
$ \cos(2\pi \left<\check \th^{j+1}, q_{N-j} \right>) < -1/2$ for all
$j=0,\dots N-1$.

\medskip

\noindent {\it Proof of Claim 2.} 
  We will treat separately the cases $\om_1\om_2<0$ and $\om_1\om_2>0$.
  
\medskip 

\noindent {\bf Case 1 : $\om_1\om_2<0$.   }   We will show that $U(r,\hat \th)$ lies in $\{r_1\geq A\} \cup \{r_2\geq A\}$, while 
 $U(r,\check \th)$ lies in $\{r_1\leq 1/A\} \cup \{r_2\leq 1/A\}$. 

 For every $j=1,\dots N$, we have the following recursive estimate:
$$
\hat r^{j}_1= \hat r^{j-1}_1 \left( 1+b \hat r^{j-1}_2 g(\hat
r^{j-1}_1)g(\hat r^{j-1}_2) \cos \left( 2\pi \left< q_j, \hat
    \theta^j \right> \right) \right) \geq \hat r^{j-1}_1.
$$ 
In the same way, $ \hat r^{j}_2 \geq \hat r^{j-1}_2$ (here we use that $q_2/q_1>0$).

Hence, if for some $i \in \{1,2\}$ and some $j<N$, $\hat r^{j}_i \geq A$, then  $\hat r^{N}_i \geq
A$. 

Assume on the contrary that for  all $j<N$, $\hat r^{j}_i < A$ for
$i=1,2$. Then, since $g\equiv 1$ inside $R(A)$, we have, using (\ref{choice_N}) that  
$$\hat r^{N}_1 \geq 
r^0_1 \prod_{j=0}^{N-1}(1+b \hat r^j_2\cos(2\pi \left< q_j, \hat  \theta^j \right>) )
\geq A^{-1} (1+b(4A)^{-1})^{N} \geq  A. 
$$
We used the fact that $R(2A)$ is invariant, so $\hat r^{j}_i \geq 1/(2A)$
for $i=1,2$ and $j=1,\dots N$.

In conclusion, $U(r,\hat \th)$ lies in $\{r_1\geq A\} \cup \{r_2\geq A\}$.

Now, for $\check \th$ the argument is similar :  For every $j=1,\dots N$
$$
\check r^{j}_1= \check r^{j-1}_1 \left(1 + b \check r^{j-1}_2 g(\check
r^{j-1}_1)g(\check r^{j-1}_2) \cos\left(2\pi \left< q_j, \check  \theta^j
\right>\right) \right) \leq \check r^{j-1}_1.
$$ 
In the same way, $ \check r^{j}_2 \leq \check r^{j-1}_2$ (here we use that $q_2/q_1>0$).

Hence, if for some $i \in \{1,2\}$ and some $j<N$, $\check r^{j}_i \leq 1/A$, then  $\check r^{N}_i \leq
1/A$. 

Assume to the contrary that for  all $j<N$, $\check r^{j}_i > 1/A$ for
$i=1,2$. Then, since $g\equiv 1$ inside $R(A)$ we have, using (\ref{choice_N}) that  
$$
\check r^{N}_1 \leq r^0_1 \prod_{j=0}^{N-1}\left( 1+b \check r^j_2\cos 
\left(  2\pi \left< q_j, \check  \theta^j \right>\right) \right)
< A (1-b (4A)^{-1} )^{N}  <  A^{-1}. 
$$
In conclusion, $U(r,\check \th)$ lies in $\{r_1\leq 1/A\} \cup \{r_2\leq 1/A\}$.

\medskip 

\noindent {\bf Case 2 : $\om_1\om_2>0$.   }  We will show that $U(r,\hat \th)$ lies in $\{r_1\geq A\} \cup \{r_2\leq 1/A\}$ while 
$U(r,\check \th)$ lies in $\{r_1\leq 1/A\} \cup \{r_2\geq A\}$. 

Contrary to {\bf case 1}, here we have that $\hat r^{j}_1$ is
increasing while 
$\hat r^{j}_2$ is decreasing (this because $q_1/q_2<0$). If for some $j<N$,  $\hat r^{j}_2 \leq 1/A$ we finish. If not, and if for some   
$j<N$,  $\hat r^{j}_1 \geq A$ we also finish. Hence we assume that for all  $j<N$,  $\hat r^{j}_1 < A$ and   $\hat r^{j}_2 > 1/A$. In such a situation and due to the fact that $g \equiv 1$ in $R(A)$ we have that 
$$\hat r^{N}_1 \geq 
r^0_1 \prod_{j=0}^{N-1}(1+b \hat r^j_2\cos(2\pi \left< q_j, \hat  \theta^j \right>) )
\geq  A^{-1} (1+b(4A)^{-1})^{N} \geq  A. 
$$

To treat $U(r,\check \th)$ we just exchange the roles of the
coordinates $r_1$ and $r_2$.

\bigskip

We proceed now to the proof of (4). 
Fix  
$p\notin \widetilde R(A)$. Consider the case
$r_1(p) < A^{-1}$, the other cases being similar. Let $\bar p$ be such that
$E_p=E_{\bar p}$ and $\bar p\in \widetilde R(A) $. The torus $\cT_{\bar
  p}$ separates the space $\widetilde E_{p}$, and the torus $\cT_{p}$ lies on
one side of it, namely the one intersecting the
set $\{ r_1 < A^{-1} \}$. Therefore, the torus $U(\cT_{\bar
  p})$ separates the space $U(\widetilde E_{p})$, having
$U(\cT_{p})$ on one side of it. Now, $U(\cT_{\bar p})$ intersects two margin
sets, so $U(\cT_{p})$ has to intersect at least one. Since  $\widetilde E_{p}$ is
close to $ U(\widetilde E_p)$, it has to intersect $\{ r_1 < A^{-1} \}$.

\medskip

We have shown that $U(\cT(p))$ 
is $\e$--close to (one or two components  of) $\widetilde E_{p}\cap M(A)$.
Finally, notice that $V\circ U(\cT(p))$
has the same property since $V=$Id
outside $\widetilde R(A)$. 

This finishes the proof of Lemma \ref{lemma_U}. \end{proof}

\section{Proofs for the case $d=4$}

The proof of Theorem \ref{theo.gdelta} follows from the 
proposition below that encloses one step of the successive conjugation scheme.
Fix $I_n=[a_{n},b_n]$ for some $n$.  Given a (large) $A$, we define
$$
I_n(A):= [a_{n}+A^{-1} ,b_n-A^{-1}]  , \quad  \widetilde I^4_n(A):= 
[A^{-1},A]^3 \times  I_n(A) \times \T^4.
$$

Define the margins set 
$$M_4(A)= \bigcup_{i=1}^3 \left(  \{r_{i}>A\} \bigcup \{r_{i}<A^{-1}\}\right).$$

\begin{prop}\label{main.liouville} Let  $n\in \N$, $\eps>0$, $s \in
  \N$
and $V \in \cU_4$ that is identity
outside $ \widetilde I^4_n(A_0)$
for some $A_0>0$, be given. Then for any $A>A_0$ there exist
$U \in \cU_4$, $T>0$ and $i\in \{1,2,3\}$ 
with $H=H_0\circ U^{-1} \circ V^{-1}$ 
satisfying the following properties:
\begin{enumerate}
\item $U={\rm Id}$ in the complement of $ \widetilde I^4_n(2A)$,

\item ${\|H -H_0\circ V^{-1} \|}_s<\eps$,

\item For any $P \in  \widetilde I^4_n(A)$  we have that
$\cO^T_H(P)$ intersects the set  $M_4(A)$;

\item Moreover, if $\o_{0,1}\o_{0,2}<0$ and if $n=3m$ for some $m \in \Z$, 
then there is a point $p'  \in  \widetilde I^4_n(A)$
such that $\cO^T_H(p')$ intersects both $\cap_{i=1}^3
\{r_{i}<2A^{-1}\}$ and $ \{r_{1}>A\}\cup  \{r_{2}>A\}$.

\end{enumerate}
\end{prop}

\begin{proof}[Proof of Theorem \ref{theo.gdelta}] 
For $n,A,T \in \N^*$, let

$$\cD(n,A,T) := \left\{  H \in  \bar{\cH_4}  \mid 
\forall P\in  \widetilde I^4_n(A),  
  \cO^T_H(P)   \text{ intersects }  M_4(A)
\right\}  $$

It is clear that $\cD(n,A,T)$ are open subsets of  
$\bar{\cH_4}  $ in any $C^s$ topology. Proposition \ref{main.liouville} implies that $ \bigcup_{T \in \N^*}   \cD(n,A,T)$
is dense in  $ \bar{\cH_4}  $ in any $C^s$ topology.
Hence the following set $\bar{\cD} \subset \cD$ is a dense $G^\delta$ set (in any $C^s$ topology)
$$
\bar{\cD} = \bigcap_{A \in \N^*} \bigcap_{n \in \N^*} 
\bigcup_{T \in \N^*}  \cD(n,A,T).
$$

In case $\o_{0,1}\o_{0,2}<0$ we add to the definition of $\cD(n,A,T)$, when $n=3m$ for some $m \in \Z$,  the existence of a point $p' \in  \widetilde I^4_n(A)$ that satisfies  $\cO^T_H(p')$ intersects both $\cap_{i=1}^3
\{r_{i}<2A^{-1}\}$ and $ \{r_{1}>A\}\cup  \{r_{2}>A\}$. The second part of Theorem \ref{theo.gdelta} then follows from the fact that if for some fixed $A>0$, we consider  $ H \in \cD(n,A,T)$ for $n$ sufficiently large, then the orbit of the  corresponding point $p' \in  \widetilde I^4_n(A)$ has its $r_4$ coordinate always smaller than $A^{-1}$. The orbit of $p'$ hence intersects  both $\cap_{i=1}^4
\{r_{i}<2A^{-1}\}$ and $ \{r_{1}>A\}\cup  \{r_{2}>A\}$. \end{proof}
 
\medskip

\begin{proof}[Proof of Proposition \ref{main.liouville}]
Since $V$ equals identity near the axes, we can, by increasing $A$,
assume without loss of generality that $V={\rm Id}$. 

Assume that $I=[a,b]=I_{3n}$, the other cases being exactly
similar. In this case 
$f_1(r_4) \equiv \bar{f}_1$ and $f_2(r_4) \equiv \bar{f}_2$ for $r_4\in
I$. Moreover, for
$\o_1=\bar{f}_1 +\omega_{0,1}$ and  $\o_2=\bar{f}_2 +\omega_{0,2}$,
the vector $(\o_1,\o_2)$ is Liouville. We will hence be able to use the two-dimensional Liouville construction of Proposition \ref{lemma.conjugacy}. 

Let $a \in C^\infty(\R)$ be such that 
$a(\xi)=0$ if $\xi \notin I(2A)$ and $a(\xi)=1$ if $\xi \in I(A)=[a+A^{-1},b-A^{-1}]$, and $\|a\|_s\leq C(s,A)$ where $C(s,A)$ is a constant that depends on $s$ and $A$ (recall that $A^{-1}$ is assumed to be small compared to the size of $I=I_{3n}$).

We construct the  map $U$ as follows.  First,
 $U$ is independent of $(r_3,\th_3)$
(i.e., $U(r,\th)=U(r_1,r_2,r_4,\th_1,\th_2,\th_4)$). Second, $r_4$ acts as a
parameter, and for each $r_4\in I(A)$ the map $U$ equals the
map provided by Lemma
\ref{lemma.conjugacy} (call the latter $U_2$, where 2 indicates the
number of degrees of freedom).
More precisely,  
in the proof of Proposition \ref{lemma.conjugacy},  $U_2$ is constructed
as a composition of a certain number of symplectic
maps 
$u_2^j: \R^2 \times \T^2$, $(  r,  \theta)\mapsto ( 
R,  \Theta) $, each one given by  a generating function of the form 
$$
S_2^j(  r_1,r_2 , \Theta_1, \Theta_2) = 
\left<( r_1,r_2) , (\Theta_1,  \Theta_2) \right> + 
g_2^j(r_1,r_2, \Theta_1, \Theta_2),
$$
where $g_2^j$ is some smooth function equal to zero in the neighborhood of the axes. 

We extend $S_2^j$ and $u_2^j$ to  $S^j$ and $u^j$ defined for $(r,\theta)\in \R^4 \times \T^4$ by letting 
$$
S^j(r, \Theta) = \left< r, \Theta\right> + a(r_4) g_2^j(r_1,r_2, \Theta_1, \Theta_2) . 
$$
Since $a \equiv 0$ on $I(2A)^c$, we get (1) of Proposition
\ref{main.liouville} 
from (1) of Proposition \ref{lemma.conjugacy}. To check (2), observe that $a(r_4)$ appears just like a parameter in the construction of $S^j$ from that of $S_2^j$. Thus, 
since  $\|a\|_s\leq C(s,A)$ we get (2) by just taking $\eps$ sufficiently small in Proposition \ref{lemma.conjugacy}. 

Now for $P \in  \widetilde I^4_n(A)$  we have that $a(r_4(P))=1$ and since $r_4$ is invariant under the flow we get that the dynamics in $(r_1,r_2,\theta_1,\theta_2)$ coordinates is exactly that of Proposition \ref{lemma.conjugacy}, hence (3) holds.

To prove (4) of the Proposition, 
choose $p' \in   \widetilde I^4_n(A)$ with $r_4$, $\th_4$, $\th_3$ arbitrary, with $r_3<2A^{-1}$ and with the projection of $p'$ on the 
$(r_1,r_2,\theta_1,\theta_2)$ coordinates being  the point $p'_2$ that satisfies (4) of Proposition \ref{lemma.conjugacy}. Clearly, (4) of Proposition \ref{main.liouville} holds for $p'$.  \end{proof}

\section{Proofs for the case $d=3$.}

Fix $I_n=[a_n, a_{n+1}]$ for some $n$. Given a (large) $A$, we define
$$
I_n(A):=[a_{n}+A^{-1},a_{n+1}-A^{-1}], \quad \widetilde I^3_n(A):= 
R(A)\times  I_n(A)\times \T^3.
$$

Define the  margins set 
$$
\begin{aligned}
M_3(A) =  \bigcup_{i=1}^3 \{r_{i}>A\} \cup \{r_{1}<A^{-1}\}  \cup
\{r_{2}<A^{-1}\} 
\cup \\
\bigcup_{n \in \Z}  \{r_3 \in I_n \setminus
I_n(A)\}. 
\end{aligned}
$$

The following proposition is an analog of Proposition \ref{main.liouville}.  
\begin{prop}\label{lem_Ud3}
Let  $n\in \N$, $\eps>0$, $s \in
  \N$
and $V \in \cU_3$ that is identity
outside $\widetilde I^3_n(A_0)$
for some $A_0>0$, be given. Then for any $A>A_0$ there exist
$U \in \cU_3$, $T>0$ and $(i_1,i_2)\in \{1,2,3\}$
distinct, with the following properties:
\begin{enumerate}
\item $U={\rm Id}$ in the complement of $\widetilde I^3_n(2A)$,

\item ${\|H_0^3\circ U^{-1} \circ V^{-1} -H_0^3\circ V^{-1} \|}_s<\eps$,

\item For any $P \in\widetilde I^3_n(A) $  we have that
$\cO^T_H(P)$ intersects $M_3(A)$. 

\item Moreover, if $\o_{0,1}\o_{0,2}<0$ and if $n=3m$ for some $m \in \Z$,
then there is a point $p'  \in  \widetilde I^3_n(A)$
such that $\cO^T_H(p')$ intersects both $\cap_{i=1}^2
\{r_{i}<2A^{-1}\}$ and $ \{r_{1}>A\}\cup  \{r_{2}>A\}$.
\end{enumerate} 
\end{prop}

\begin{proof}[Proof of Theorem \ref{theo.gdelta_d3}] 
The proof follows exactly the same lines as the proof of  Theorem \ref{theo.gdelta}. 

For $n,A,T \in \N^*$, we let 
$$\cD(n,A,T) := \left\{  H \in  \bar{\cH_0}  \mid 
\forall P\in  \widetilde I^3_n(A),  
  \cO^T_H(P)   \text{ intersects }  M_3(A)
\right\} . 
$$
and we see that  $\bar{\cD} \subset \cD$  given by 
$$
\bar{\cD} = \bigcap_{A \in \N^*} \bigcap_{n \in \N^*} 
\bigcup_{T \in \N^*}  \cD(n,A,T).
$$
is a dense $G^\delta$ set (in any $C^s$ topology) in $ \bar{\cH_0}$.

In case $\o_{0,1}\o_{0,2}<0$ we add to the definition of $\cD(n,A,T)$ the existence of a point $p'$ satisfying conclusion (4) of Proposition \ref{lem_Ud3}. 
The second part of Theorem \ref{theo.gdelta} then follows from the fact that if for some fixed $A>0$, we consider  $ H \in \cD(n,A,T)$ for $n$ sufficiently large, then the orbit of the  corresponding point $p' \in  \widetilde I^3_n(A)$ has its $r_3$ coordinate always smaller than $A^{-1}$ since it lies in $I_n$.  The orbit of $p'$ hence intersects  both $\cap_{i=1}^3
\{r_{i}<2A^{-1}\}$ and $ \{r_{1}>A\}\cup  \{r_{2}>A\}$. \end{proof}
\medskip
\begin{proof}[Proof of Proposition \ref{lem_Ud3}] \  The proof is
  similar to that of  
Proposition \ref{main.liouville} (which relies on Lemma
\ref{lemma.conjugacy}). We shall only describe the
modifications that have to be done in order to get
the conjugacy $U$ for $d=3$.

When $r_3\in I_{3n}$, then $(\om_1,\om_2)$ is a constant Liouville
vector and the construction of Proposition \ref{lemma.conjugacy} is
carried out in the $(r_1,r_2,\th_1,\th_2)$-space, with $r_3$ acting as a parameter exactly as $r_4$ acted as a parameter in the proof of Proposition \ref{main.liouville}. 

The situation for $r_3 \in I_n$ for $n={3m+1}$ or $n={3m+2}$ is slightly different since $r_3$ is not invariant anymore. Suppose that $n=3m+1$, the other case being similar. For $r_3 \in I_n$ we have that 
 the vector $(\o_1,\o_3)$ is constant and Liouville.  The idea is that since $r_3$ will diffuse but remain inside $I_n$, one can still perform the construction of Proposition \ref{lemma.conjugacy} with $r_3$ playing in the same time the role of a parameter and that of a diffusing action coordinate. The diffusion in the $r_3$ variable  is thus limited to accumulating the set $  \{r_3 \in I_n \setminus I_n(A)\}.$ All the rest of the proof is similar to the proof of Proposition \ref{main.liouville}.
 \end{proof}

\end{document}